\documentclass[UTF8]{article}

\UseRawInputEncoding

\usepackage{color}
\usepackage{epsfig}
\usepackage{subfig}
\usepackage{mathrsfs}
\usepackage{booktabs}
\usepackage{multirow}
\usepackage{amsmath,amssymb}
\usepackage{amsthm}
\usepackage{fancyhdr}
\usepackage{tabularx}
\usepackage{xfrac}
\usepackage{mathptmx}
\usepackage{mathrsfs}
\usepackage{graphicx,epsfig}
\usepackage{times}

\usepackage[noend]{algpseudocode}
\usepackage{algorithmicx,algorithm}

\newtheorem{theorem}{Theorem}[section]

\newtheorem{defn}{Definition}[section]
\newtheorem{coro}{Corollary}[section]
\newtheorem{prop}{Proposition}[section]
\newtheorem{lemma}{Lemma}[section]
\newtheorem{remark}{Remark}[section]

\begin{document}

\title{Distributed Coverage Control of Multi-Agent Systems with Load Balancing in Non-convex Environments}

\author{Chao Zhai, Pengyang Fan~\thanks{Chao Zhai and Pengyang Fan are with School of Automation, China University of Geosciences, Wuhan 430074 China, and with Hubei Key Laboratory of Advanced Control and Intelligent Automation for Complex Systems and Engineering Research Center of Intelligent Technology for Geo-exploration, Ministry of Education, Wuhan 430074 China. Email: zhaichao@amss.ac.cn}}

\maketitle

\begin{abstract}
It is always a challenging task to service sudden events in non-convex and uncertain environments, and multi-agent coverage control provides a powerful theoretical framework to investigate the deployment problem of mobile robotic networks for minimizing the cost of handling random events. Inspired by the divide-and-conquer methodology, this paper proposes a novel coverage formulation to control multi-agent systems in the non-convex region while equalizing the workload among subregions. Thereby, a distributed coverage controller is designed to drive each agent towards the desired configurations that minimize the service cost by integrating with the rotational partition strategy. In addition, a circular search algorithm is proposed to identify optimal solutions to the problem of lowering service cost. Moreover, it is proved that this search algorithm enables to approximate the optimal configuration of multi-agent systems with the arbitrary small tolerance. Finally, numerical simulations are implemented to substantiate the efficacy of proposed coverage control approach.
\end{abstract}

Keywords: coverage control, multi-agent systems, equitable partition, distributed coordination

\section{Introduction}

The rapid development of semiconductor devices and communication technologies makes it possible to manufacture plenty of low-cost smart sensors or mobile robots, which can find applications in different fields. In contrast to the limited capability of a single sensor or robot, a group of wireless sensors or mobile robots are able to implement a variety of complicated coordination missions, such as wilderness search and rescue~\cite{mac14}, patrolling along the border~\cite{far17}, region coverage~\cite{auto13,cor04,song20}, environment monitoring~\cite{ctt21}, and missile interception~\cite{jgcd16}, to name just a few. The successful fulfillment of the above missions requires the design of cooperative control approaches for multi-agent systems. As a result, multi-agent coordination control has attracted lots of researchers in various fields in the past decades.

As a typical coordination behavior, multi-agent coverage centers on the design of control strategies that enable agents to cooperatively visit every point of interest in the given region or the coverage rate of a given terrain to guarantee no area left in the coverage region, and meanwhile optimize the index of coverage performance. In multi-robot systems, cooperative coverage problem can be classified into three modes: sweep coverage, barrier coverage and blanket coverage~\cite{gage92,book21}.  Nevertheless, environmental uncertainties and distributed coordination have made it a challenging issue for researchers to develop an effective cooperative control approach to fulfilling the coverage tasks in short time periods, while maintaining a desired level of coverage quality.

The environmental uncertainties can be handled via the scheme of region partition, and Voronoi partition and equitable workload partition have been widely investigated in the current studies. For example, a static coverage optimization problem with area constraints is considered in~\cite{cort10}, where the move-to-center-and-compute-weight strategy is designed to drive multi-agent systems towards the set of center generalized Voronoi configurations while optimizing the coverage quality. By taking into account time as the proximity metric, a Voronoi-like partition is conducted for spatiotemporal coverage in a drift field, which allows to correlate each agent with arbitrary points in its sub-region in light of the minimum time-to-go~\cite{bak13}. In order to allow for  accurate environment monitoring at varying resolution, a probabilistic spatial sensing model is proposed for sensor networks with adjustable sensing range by introducing the statistical distance, which extending the  the standard Voronoi-based coverage control law~\cite{ars19}. To deal with the changing boundaries, a  Voronoi-based blanket coverage method is developed for sensor networks to adapt to the varying coverage region by incorporating the boundary dynamics into the control law~\cite{abb19}. In addition, the geodesic Voronoi partition is employed to deploy a team of agents in the environments involving mixed-dimensional and hybrid cases~\cite{liu21}. In spite of the advantages in distributed implementation, Voronoi partition is subject to high computation costs in determining sub-region boundaries and centroids.  As a result, a distinct coverage formulation based on equitable workload partition is proposed by dividing the coverage region into multiple stripes and further partitioning each stripe into sub-stripes with the same workload~\cite{auto13,tsmc21}. Therein, each agent only needs to complete the workload in its own sub-stripe for minimizing the coverage time. Since the optimal coverage time is unavailable due to uncertainties, the error between the actual coverage time and the optimal time is estimated. The equitable partition policies are particularly appealing to researchers due to its ubiquitous applications and ease of analysis~\cite{pav11}.
Nevertheless, existing studies based on equitable workload partition mainly deals with coverage problem of regular regions (e.g., rectangular area, convex polygon, regions with parallel boundaries) and can not directly be applied to the irregular non-convex environment. Moreover, these partition algorithms either have to predetermine the sub-region weights or are subject to high computation costs and even the losing of subregion connectivity~\cite{pal19}.
Besides, existing coverage algorithms neither guarantee the optimality of solutions in theory nor provides a quantitative assessment of coverage performance as compared to the optimal configuration.

To address the foregoing issues, this paper aims to develop a new partition-based formulation for efficient coverage in non-convex uncertain environments. In brief, the core contributions of this work are listed as follows.
\begin{enumerate}
\item Develop a novel coverage formulation of multi-agent systems for minimizing the service cost in non-convex uncertain environment.
\item Design a distributed control strategy to drive each agent towards the desired location while reconciling the conflict with load balancing.
\item Propose a circular search algorithm to identify the optimal configuration of multi-agent systems with theoretical guarantee on the optimality of service cost.
\end{enumerate}

The remainder of this paper is organized as follows. Section~\ref{sec:prob} formulates the coverage control problem of multi-agent systems in uncertain environment. Section~\ref{sec:main} provides the distributed control algorithm and technical analysis.  Section~\ref{sec:cas} presents simulation results to validate the proposed control strategy. Section~\ref{sec:con} concludes this paper and discusses future work.

\section{Problem Formulation}\label{sec:prob}
This section formulates the coverage problem of multi-agent systems for optimal monitoring in uncertain environments.
Consider a two-dimensional annular region $\Omega$, which is enclosed by two continuous closed curves (see Fig.~\ref{reg}),
and the inner and outer curves are described by polar equations $r_{in}(\theta)$ and $r_{out}(\theta)$ with respect to the origin $O$, respectively.
Here, $\theta\in[0,2\pi)$ refers to the phase angle at a point on the curve.
Then the mathematical expression of coverage region $\Omega$ is given by $\Omega=\{{({r,\theta})|{r_{in}}(\theta) \le r \le {r_{out}}(\theta)}\}$.
Suppose that two regions enclosed by $r_{in}(\theta)$ and $r_{out}(\theta)$ are star-shaped with respect to a common reference point, respectively. The definition of star-shaped set is given below~\cite{tu11}.

\begin{defn}\label{star}
A subset $\Omega\subset R^2$ is star-shaped with respect to a reference point $\mathbf{O}\in\Omega$ if for
every $\mathbf{p}\in\Omega$, the line segment from $\mathbf{p}$ to $\mathbf{O}$ lies in $\Omega$.
\end{defn}

\begin{figure}
\scalebox{0.1}[0.1]{\includegraphics{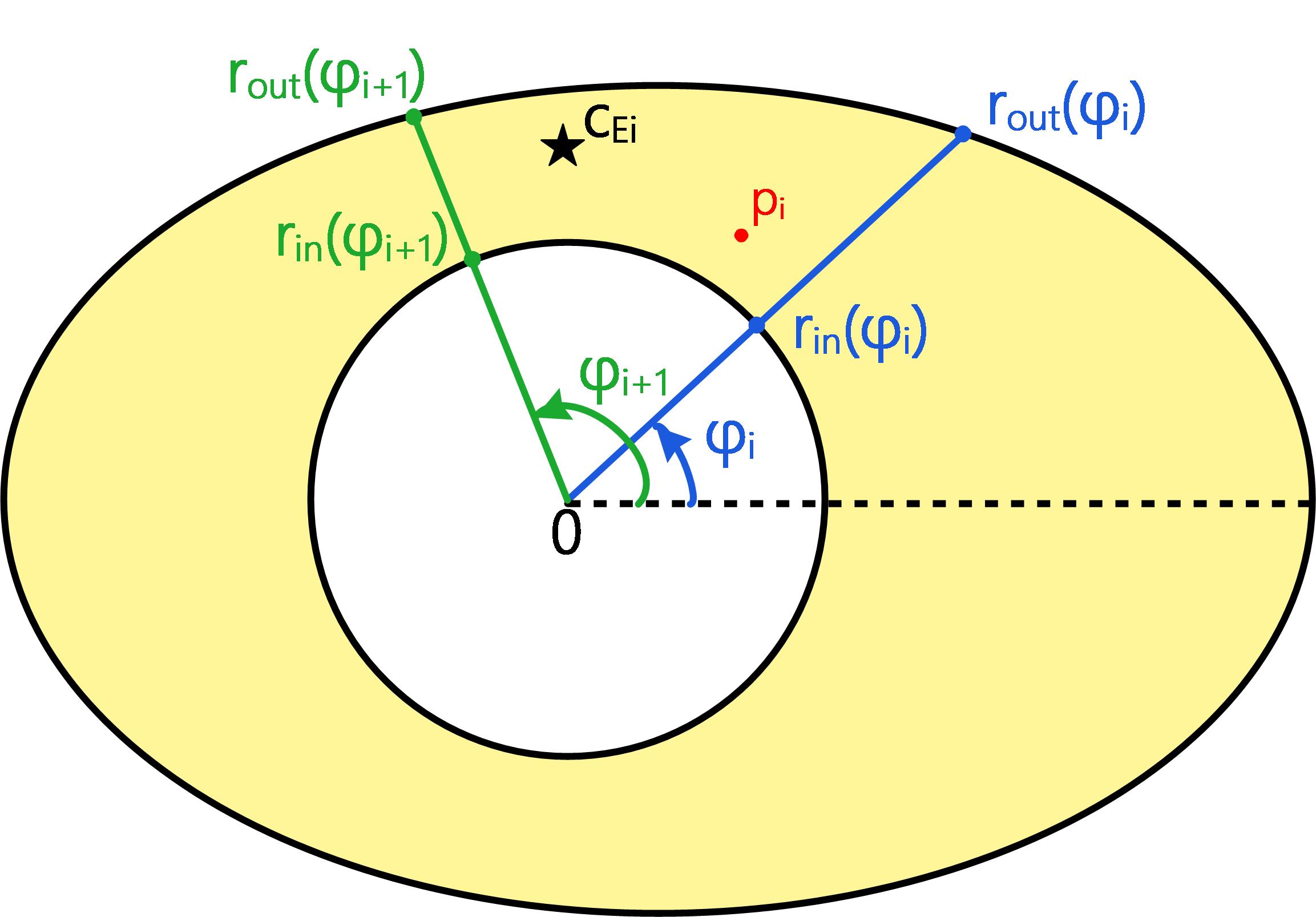}}\centering
\caption{\label{reg} Illustration of non-convex coverage region with a hole inside. The coverage is divided into multiple subregions by partition bars, and the origin $\mathbf{O}$ refers to a common reference point for the inner and outer star-shaped sets. The black star denotes the centroid of subregion, while the red dot represents the position of agents.}
\end{figure}


The coverage region can be approximated by polygons, and it has been proved that the problem of testing a star-shaped polygon and identifying a reference point can be solved in linear time~\cite{mat96}.
Define a density function $\rho: \Omega\mapsto R^{+}$ with $\rho(r,\theta)\in[\underline{\rho},\bar{\rho}]$, and it can characterize the importance of a certain point in the region or the amount of information or workload at a certain point of the region. Each agent is equipped with a virtual partition bar that is represented
by $\{ ( {r,{\varphi _i}}) \subset \Omega~|~{r_{in}}({{\varphi _i}}) \le r \le {r_{out}}( {{\varphi _i}})\}$, $i\in I_N=\{ {1,2,...,N}\}$, where ${\varphi_i} \in [ {0,2\pi })$ is the phase angle of partition bar for the $i$-th agent. The partition bars are numbered sequentially as follows $0\leq\varphi_1(0) <\varphi_2(0)<\cdot\cdot\cdot<\varphi_N(0)<2\pi$, where $\varphi_i(0)$ represents the initial phase of partition bar. Thus, a group of $N$ agents enable to divide the coverage region $\Omega$ into $N$ sub-regions, i.e. $\Omega=\bigcup_{i=1}^{N}E_i$, where $E_i$ denote the phase angle of partition bar for the $i$-th agent, and it is enclosed by inner and outer curves as well as partition bars $\varphi_{i}$ and $\varphi_{i+1}$. In addition, the position of the $i$-th agent is denoted by ${p}_i\in\Omega$, $i\in I_N$. The area in charge of each agent is represented as follows
\begin{equation*} 
E_i(\varphi)=\left\{ {\begin{array}{*{20}{c}}
{\left\{ {(r,\theta)\in\Omega~|~\varphi_i < \theta < 2\pi,~ 0\le\theta <\varphi_{i + 1}} \right\},}&{\begin{array}{*{20}{c}}&{if~\varphi_{i + 1} <\varphi_i}
\end{array}}\\
{\left\{ {(r,\theta)\in\Omega~|~\varphi_i < \theta < \varphi_{i+1}}\right\}.}&{\text{otherwise}}
\end{array}}\right.
\end{equation*}
where the function $\omega(\theta)$ is given by $\omega(\theta)=\int_{{{r}_{in}}(\theta)}^{{{r}_{out}}(\theta)}{\rho( r,\theta)rdr}$. As a result, the workload on the $i$-th subregion can be computed by
\begin{equation}\label{mi}
     {m_i} = \left\{
      {\begin{aligned}
       &{\int_{{\varphi_i}}^{2\pi } \omega  (\theta )d\theta  + \int_0^{{\varphi _{i + 1}}} \omega  (\theta )d\theta ,}& \quad & if~{\varphi_{i+1}<\varphi_i}\\
	   &{\int_{{\varphi_i}}^{\varphi_{i+1}}\omega (\theta)d\theta.} & \quad & \text{otherwise} \\
\end{aligned}} \right.
\end{equation}
with $\varphi_{N+1}=\varphi_1$. To balance workload among subregions, the dynamics of partition bar is designed as
\begin{equation}\label{sys}
\dot{\varphi}_i = \kappa_{\varphi}(m_i-m_{i-1}), \quad i\in I_N
\end{equation}
with the positive constant ${\kappa_\varphi}$ and $m_0=m_N$. Note that the evolution of $\varphi_i$ may exceed the range of interval $[0,2\pi)$.
Since phase angle $\varphi_i$ wraps around while reaching $2\pi$, the modulo operation is adopted with $\varphi_i=\text{mod}(\varphi_i,2\pi)$
in this work to ensure $\varphi_i\in[0,2\pi)$, $i\in I_N$.

Inspired by the work in \cite{cor04}, the performance index of multi-agent coverage problem is designed as follows
\begin{equation}\label{cost}
J(\mathbf{\varphi},\mathbf{p})= \sum_{i=1}^N \int_{E_i(\mathbf{\varphi})}f({p}_i,q)\rho(q)dq.
\end{equation}
with $\mathbf{p}=({p}_1,{p}_2,...,{p}_N)^T$ and $\mathbf{\varphi}=(\varphi_1,\varphi_2,...,\varphi_N)^T$. In addition, $f({p}_i,q)$ is continuously differentiable with respect to ${p}_i$, and it is used to quantify the cost of moving from ${p}_i$ to $q$ to service the event. Thus, $f({p}_i,q)$ is closely dependent on the distance between ${p}_i$ and $q$. The dynamics of multi-agent system is presented as
\begin{equation}\label{dp}
    \dot{p}_i = {u}_i,\quad i\in I_N
\end{equation}
with the control input $\mathbf{u}_i$. The performance index characterizes the total cost of multi-agent systems to service the events on the region $\Omega$,
and a lower value of performance index (\ref{cost}) indicates a better deployment of agents and partition bars. Essentially, our goal is to design a distributed control algorithm to solve the following optimization problem
\begin{equation}\label{min_cost}
\min_{\mathbf{\varphi},\mathbf{p}} J(\mathbf{\varphi},\mathbf{p})
\end{equation}
by deploying agents and equalizing the workload among subregions with partition bars.


\section{Main Results}\label{sec:main}
This section presents a distributed coverage algorithm, and theoretical analysis is conducted as well. Firstly, Some key lemmas are provided,
followed by a distributed search algorithm to find optimal solutions to Problem (\ref{min_cost}) with a specified tolerance. Then technical
analysis is presented to work out the conditions that guarantee the accuracy of the proposed algorithm. Finally, theoretical results are
applied to the classic location optimization problem.

\begin{lemma}\label{lem_equ}
Partition dynamics (\ref{sys}) equalizes the workload on each subregion.
\end{lemma}

\begin{proof}
Construct the following Lyapunov function
\[{V}(\mathbf{\varphi}) = \frac{1}{2}\sum\limits_{i=1}^N\left(m_i-\bar{m}\right)^2,
\quad \bar m = \frac{1}{N}{\int_0^{2\pi } {\omega \left( \theta  \right)d\theta }}\]
The time derivative of Lyapunov function along the trajectory of the system is given as follows
\[\begin{aligned}
\frac{dV(\mathbf{\varphi})}{dt}= \sum\limits_{i = 1}^N {\left( {{m_i} - \bar m} \right){{\dot m}_i}} &=\sum\limits_{i = 1}^N {{m_i}{{\dot m}_i}}-\bar{m}\sum_{i=1}^{N}\dot{m}_i \\
&= \sum\limits_{i = 1}^N {{m_i}\left( {{{\dot \varphi }_{i + 1}}\omega \left( {{\varphi _{i + 1}}} \right) - {{\dot \varphi }_i}\omega \left( {{\varphi _i}} \right)} \right)} \\
&=-\sum\limits_{i = 1}^N {\left( {{m_i} - {m_{i - 1}}} \right){{\dot \varphi }_i}\omega \left({{\varphi_i}} \right)} \\
&=-\kappa_{\varphi}\sum_{i=1}^N(m_i-m_{i-1})^2\omega\left({{\varphi _i}}\right)
\end{aligned}\]
Since both $\kappa_\varphi>0$ and $\omega \left({{\varphi_i}}\right)>0$, this implies that $\dot{V}(\mathbf{\varphi})<0$. Therefore, the state of the system will approach the equilibrium point, which means that $m_1(t)=m_2(t)=...=m_N(t)$, as time goes to the infinity.
\end{proof}

\begin{remark}
Let $\bar{\varphi}(t)=\frac{1}{N}\sum_{i=1}^{N}\varphi_i(t)$ represent the average phase of all partition bars. The time derivative of $\bar{\varphi}(t)$ is given by
$$
\frac{d\bar{\varphi}(t)}{dt}=\frac{1}{N}\sum_{i=1}^{N}\dot{\varphi}_i(t)
=\frac{\kappa_{\varphi}}{N}\sum_{i=1}^{N}\left(m_i-m_{i-1}\right)=0,
$$
which implies that the average phase of all partition bars keeps unchanged during the entire partition process.
\end{remark}


\begin{lemma}\label{lem_comp}
$$
\sum_{i=1}^{N}(m_i-m_{i-1})^2\geq\frac{2\lambda_{\min}(S)V(\mathbf{\varphi})}{N},
$$
where $\lambda_{\min}(S)$ denotes the minimum eigenvalue of the positive definite matrix $S$ as follows
\begin{equation*}
S=\left(
         \begin{array}{cccccccc}
           6 & 2 & 3 & . & . & 3 & 3 & 4 \\
           2 & 4 & 1 & . & . & 2 & 2 & 3 \\
           3 & 1 & 4 & . & . & 2 & 2 & 3 \\
           . & . & . & . & . & . & . & . \\
           . & . & . & . & . & . & . & . \\
           3 & 2 & 2 & . & . & 4 & 1 & 3 \\
           3 & 2 & 2 & . & . & 1 & 4 & 2 \\
           4 & 3 & 3 & . & . & 3 & 2 & 6 \\
         \end{array}
       \right)\in R^{(N-1)\times(N-1)}, \quad~N>2
\end{equation*}
and $S=8$ for $N=2$.
\end{lemma}
\begin{proof}
For simplicity, introduce the error variable $e_i=m_i-\bar{m}$, $i\in I_N$ and the vector $\mathbf{e}=(e_1,e_2,...,e_{N-1})^T$. It follows from $\sum_{i=1}^{N}e_i=0$ that
\begin{equation*}
\begin{split}
\sum_{i=1}^{N}(m_i-m_{i-1})^2&=\sum_{i=1}^{N}(e_i-e_{i-1})^2 \\
&=(e_1-e_{N})^2+\sum_{i=2}^{N-1}(e_i-e_{i-1})^2+(e_N-e_{N-1})^2 \\
&=\left(e_1+\sum_{i=1}^{N-1}e_{i}\right)^2+\sum_{i=2}^{N-1}(e_i-e_{i-1})^2+\left(\sum_{i=1}^{N-1}e_{i}+e_{N-1}\right)^2 \\
&=\mathbf{e}^TS\mathbf{e}\geq\lambda_{\min}(S)\|\mathbf{e}\|^2.
\end{split}
\end{equation*}
Considering that
$$
V(\mathbf{\varphi})=\frac{1}{2}\sum_{i=1}^N\left(m_i-\bar{m}\right)^2=\frac{1}{2}\sum_{i=1}^N(e_i)^2=\frac{1}{2}\|\mathbf{e}\|^2+\frac{1}{2}\left(\sum_{i=1}^{N-1}e_i\right)^2\leq \frac{N}{2}\|\mathbf{e}\|^2,
$$
one gets
$$
\sum_{i=1}^{N}(m_i-m_{i-1})^2\geq\lambda_{\min}(S)\|\mathbf{e}\|^2\geq\frac{2\lambda_{\min}(S)V(\mathbf{\varphi})}{N}.
$$
This completes the proof.
\end{proof}

\begin{lemma}\label{lem_exp}
There exist positive constants $c_1$ and $c_2$ such that
$|m_i(t)-m_{i-1}(t)|\leq c_1 e^{-c_2t}$, $\forall t\geq0$.
\end{lemma}
\begin{proof}
According to Lemma~\ref{lem_equ}, the time derivative of $V$ along the system (\ref{sys}) is given by
$$
\dot{V}=-\kappa_{\varphi}\sum_{i=1}^N(m_i-m_{i-1})^2\omega\left({{\varphi _i}}\right)\leq-\kappa_{\varphi}\omega_{\min}\sum_{i=1}^N(m_i-m_{i-1})^2
$$
with $\omega_{\min}=\min_{\varphi\in[0,2\pi]}\omega(\varphi_i)$. It follows from Lemma~\ref{lem_comp} that
$$
\dot{V}\leq-\kappa_{\varphi}\omega_{\min}\sum_{i=1}^N(m_i-m_{i-1})^2\leq -\frac{2\kappa_{\varphi}\omega_{\min}\lambda_{\min}(S)}{N}V.
$$
Solving the above differential inequality yields
$$
V(t)\leq V(0)e^{-\frac{2\kappa_{\varphi}\omega_{\min}\lambda_{\min}(S)}{N}t}
$$
and $|m_i(t)-m_{i-1}(t)|\leq\sqrt{2V(t)}\leq c_1e^{-c_2 t}$
with $c_1=\sqrt{2V(0)}$ and $c_2=\frac{\kappa_{\varphi}\omega_{\min}\lambda_{\min}(S)}{N}$. The proof is thus completed.
\end{proof}

\begin{lemma}\label{lem_col}
Collision avoidance of split bars is guaranteed with partition dynamics (\ref{sys}).
\end{lemma}

\begin{proof}
First of all, one will obtain the equivalent condition on the collision of split bars.
For $i\in I_N$, if the split bars of Agent $i$ and Agent $i+1$ are collided,
one gets $\varphi_i=\varphi_{i+1}$ and thus $m_i=0$ due to $m_i=\int_{\varphi_i}^{\varphi_{i+1}}\omega(\theta)d\theta$.
On the other hand, if $m_i=0$, one obtains
$$
\omega_{\min}(\varphi_{i+1}-\varphi_i)\leq m_i=\int_{\varphi_i}^{\varphi_{i+1}}\omega(\theta)d\theta=0\leq\omega_{\max}(\varphi_{i+1}-\varphi_i),
$$
which indicates $\varphi_{i+1}-\varphi_i\leq m_i/\omega_{\min}=0$ and $\varphi_{i+1}-\varphi_i\geq m_i/\omega_{\max}=0$ as a result of
$$
\omega_{\min}=\min_{\theta\in(0,2\pi]}\omega(\theta)>0, \quad \omega_{\max}=\max_{\theta\in(0,2\pi]}\omega(\theta)>0
$$
This implies $\varphi_i=\varphi_{i+1}$ and the collision of split bars from Agent $i$ and Agent $i+1$. Thus, the collision of split bars (i.e., $\varphi_i=\varphi_{i+1}$) is equivalent to $m_i=0$, $i\in I_N$. In what follows, one will prove that $m_i(t)>0$, $\forall t>0$ and $i\in I_N$ for the dynamics (\ref{sys}). Let $i^{*}_t=\arg\min_{i\in I_N}m_i(t)$ and it follows from $m_{i^{*}_t}=\min_{i\in I_N}m_i(t)$ that
\begin{equation*}
\begin{split}
\dot{m}_{i^{*}_t}&=\omega(\varphi_{i^{*}_t+1})\dot{\varphi}_{i^{*}_t+1}(t)-\omega(\varphi_{i^{*}_t})\dot{\varphi}_{i^{*}_t}(t) \\
&=\omega(\varphi_{i^{*}_t+1})\kappa_{\varphi}\left(m_{i^{*}_t+1}-m_{i^{*}_t}\right)-\omega(\varphi_{i^{*}_t})\kappa_{\varphi}\left(m_{i^{*}_t}-m_{i^{*}_t-1}\right) \\
&=\omega(\varphi_{i^{*}_t+1})\kappa_{\varphi}\left(m_{i^{*}_t+1}-m_{i^{*}_t}\right)+\omega(\varphi_{i^{*}_t})\kappa_{\varphi}\left(m_{i^{*}_t-1}-m_{i^{*}_t}\right)\geq 0,
\end{split}
\end{equation*}
which indicates
$$
{m}_{i^{*}_t}(t)={m}_{i^{*}_t}(0)+\int_{0}^{t}\dot{m}_{i^{*}_t}(t)dt>0, \quad \forall t>0
$$
due to ${m}_{i^{*}_t}(0)>0$. Considering that $m_i(t)\geq{m}_{i^{*}_t}(t)$, $\forall i\in I_N$ and $t>0$, it follows that $m_i(t)>0$,$\forall i\in I_N$ and $t>0$, which guarantees the collision avoidance of split bars.
\end{proof}


\begin{lemma}\label{dce}
$$
\lim_{t\to\infty}\dot{c}_{E_i}(t)={0}, \quad i\in I_N.
$$
\end{lemma}

\begin{proof}
It follows from Lemma \ref{lem_equ} and Equations (\ref{sys}) that $\lim_{t\rightarrow\infty}\dot{\varphi}_i(t)=0$, $i\in I_N$. In light of Equation (\ref{mi}), one has
$$
\dot{m}_i(t)=\omega(\varphi_{i+1})\dot{\varphi}_{i+1}(t)-\omega(\varphi_{i})\dot{\varphi}_{i}(t),
$$
which leads to $\lim_{t\to\infty}\dot{m}_i(t)=0$. For simplicity, introduce the following notation
\begin{equation*}
{m_{i,x}} = \left\{
{\begin{aligned}
&{\int_{{\varphi _i}}^{2\pi } \omega_x(\theta )d\theta  + \int_0^{{\varphi _{i + 1}}} \omega_x(\theta )d\theta ,}& \quad & if~{{\varphi _{i + 1}} < {\varphi _i}}\\
&{\int_{{\varphi _i}}^{{\varphi _{i + 1}}} \omega_x (\theta )d\theta .}& \quad & otherwise \\
\end{aligned}} \right.
\end{equation*}
and one has
$$
\dot{m}_{i,x}(t)=\omega_x(\varphi_{i+1})\dot{\varphi}_{i+1}(t)-\omega_x(\varphi_{i})\dot{\varphi}_{i}(t),
$$
which leads to $\lim_{t\to\infty}\dot{m}_{i,x}(t)=0$. It follows from $\lim_{t\to\infty}\dot{m}_i(t)=\lim_{t\to\infty}\dot{m}_{i,x}(t)=0$ that
$$
\lim_{t\to\infty}\dot{c}^x_{E_i}(t)=\lim_{t\to\infty}\frac{\dot{m}_{i,x}(t)\cdot m_i(t)-m_{i,x}(t)\cdot \dot{m}_i(t)}{\left(m_i(t)\right)^2}=0.
$$
Similarly, one can obtain $\lim_{t\to\infty}\dot{c}^y_{E_i}(t)=0$. Considering that $\dot{c}_{E_i}=(\dot{c}^x_{E_i},\dot{c}^y_{E_i})^T$, one gets $\lim_{t\to\infty}\dot{c}_{E_i}(t)={0}$, $i\in I_N$.
The proof is thus completed.
\end{proof}

\begin{lemma}\label{lem_phi}
$$
\lim_{t\rightarrow\infty}\varphi_i(t)=\varphi^{*}_i, \quad i\in I_N
$$
with the constant $\varphi^{*}_i$.
\end{lemma}

\begin{proof}
It follows from Lemma~\ref{lem_equ} and Equation (\ref{sys}) that
$$
\lim_{t\rightarrow\infty}\dot{\varphi}_i(t)=\lim_{t\rightarrow\infty}\kappa_{\varphi }\left[m_i(t)-m_{i-1}(t)\right]=0.
$$
In light of Lemma~\ref{lem_exp}, there exist positive constants $c_1$ and $c_2$ such that
$$
\dot{\varphi}_i(t)\leq|\dot{\varphi}_i(t)|=\kappa_{\varphi}|m_i(t)-m_{i-1}(t)|\leq\kappa_{\varphi}c_1 e^{-c_2t},
$$
which leads to
$$
\lim_{t\to\infty}{\varphi}_i(t)-{\varphi}_i(0)= \int_{0}^{+\infty}\dot{\varphi}_i(t)dt\leq\int_{0}^{+\infty}|\dot{\varphi}_i(t)|dt\leq\kappa_{\varphi}c_1 \int_{0}^{+\infty}e^{-c_2t}dt=\frac{\kappa_{\varphi}c_1}{c_2}<+\infty.
$$
According to Theorem $10.33$ in~\cite{apo74}, $\dot{\varphi}(t)$ is improper
Riemann-integrable on $[0,+\infty)$, and thus the improper integral $\int_{0}^{+\infty}\dot{\varphi}_i(t)dt$ exists, which implies the existence of $\lim_{t\to\infty}{\varphi}_i(t)$ because of the following equality
$$
\lim_{t\to\infty}{\varphi}_i(t)={\varphi}_i(0)+\int_{0}^{+\infty}\dot{\varphi}_i(t)dt.
$$
Therefore, $\varphi_i(t)$, $i\in I_N$ converges to the constant value $\varphi^{*}_i$ as time $t$ goes to the positive infinity.
\end{proof}

\begin{prop}\label{lem_solex}
For some $\varphi_i$, $i\in I_N$, there exist solutions to $\nabla_{{p}_i}J={0}$ on the domain
$\mathrm{D}=[\underline{a},\bar{a}]\times[\underline{b},\bar{b}]$ if
\begin{equation}\label{inner}
\langle\nabla_{{p}_i}J(\mathbf{A}\mathbf{z}+\mathbf{b}),\mathbf{p}\rangle>0, \quad \forall \mathbf{z}\in\partial I^2,
\end{equation}
where $\mathbf{z}=(p_x,p_y)$, $\mathbf{A}=\frac{1}{2}\text{diag}(\bar{a}-\underline{a},\bar{b}-\underline{b})$,
$\mathbf{b}=\frac{1}{2}(\bar{a}+\underline{a},\bar{b}+\underline{b})^T$, $I^2=[-1,1]\times[-1,1]$, and
$$
\nabla_{{p}_i}J=\int_{E_i(\varphi)}\frac{\partial f({p}_i,q)}{\partial {p}_i}\rho(q)dq
$$
denotes the gradient of cost function (\ref{cost}) with respect to ${p}_i$. In addition,
$\partial I^2$ denotes the boundary of $I^2$, and the symbol $\langle\cdot,\cdot\rangle$ represents the inner product
in Euclidean space.
\end{prop}

\begin{proof}
Construct a two-dimension continuous function of two variables as follows
$$
{F}_i(\mathbf{z})=\nabla_{{p}_i}J(\mathbf{A}\mathbf{z}+\mathbf{b}), \quad \mathbf{z}=(p_x,p_y) \in [-1,1]\times[-1,1].
$$
Then the existence of solutions to $\nabla_{{p}_i}J={0}$ on $\mathrm{D}=[\underline{a},\bar{a}]\times[\underline{b},\bar{b}]$ is guaranteed if solutions to ${F}_i(\mathbf{z})={0}$ are available on $I^2$. According to the generalization of Bolzano's Theorem~\cite{mora02}, there exist solutions to ${F}_i(\mathbf{z})={0}$ on $I^2$ if
$\langle{F}_i(\mathbf{z}),\mathbf{z}\rangle>0$, $\forall\mathbf{z}\in\partial I^2$, which is equivalent to $\langle\nabla_{{p}_i}J(\mathbf{A}\mathbf{z}+\mathbf{b}),\mathbf{z}\rangle>0$, $\forall \mathbf{z}\in\partial I^2$. The proof is thus completed.
\end{proof}

\begin{remark}
The subregion $E_i$ can be decomposed into a series of squares $\mathrm{D}_k=[\underline{a}_k,\bar{a}_k]\times[\underline{b}_k,\bar{b}_k]$ satisfying $\bigcup_kD_k\subseteq E_i$. Then the existence of solutions to $\nabla_{{p}_i}J={0}$ on each $D_k$ can be identified with
the inequality (\ref{inner}).
\end{remark}


\begin{theorem}\label{theo1}
For multi-agent dynamics (\ref{dp}) with control input
\begin{equation}\label{u_star}
{u}_i=-\kappa_p\cdot({p}_i-{p}_i^{*})
\end{equation}
with ${p}_i^{*}=\arg\inf_{{p}_i\in\Pi_i}J_{E_i}$ and $\Pi_i=\{{p}_i\in E_i|\nabla_{{p}_i}J={0}\}\bigcup\partial E_i$
and partition dynamics (\ref{sys}), the following claims hold.  \\
1) If $\text{rank}[\mathbf{H}_{J,p_i}(p_i^{*})]=2$, $\lim_{t\to\infty}\dot{p}^{*}_i(t)={0}$, $\forall~i\in I_N$; \\
2) If $\text{rank}[\mathbf{H}_{J,p_i}(p_i^{*})]=2$, $\lim_{t\to\infty}\left\|\nabla_{{p}_i}J\right\|=0$, $\forall~i\in I_N$; \\
3) Equitable workload partitions can be achieved with exponential convergence rate.
\end{theorem}

\begin{proof}
For Claim 1), it follows from $\nabla_{{p}_i}J(p_i^{*})={0}$ that
$$
\frac{d}{dt}\nabla_{{p}_i}J(p_i^{*})=\mathbf{H}_{J,p_i}(p_i^{*})\cdot\dot{p}_i^{*}+\frac{\partial\nabla_{{p}_i}J(p_i^{*})}{\partial\varphi_{i+1}}\dot{\varphi}_{i+1}
+\frac{\partial\nabla_{{p}_i}J(p_i^{*})}{\partial\varphi_{i}}\dot{\varphi}_{i}={0}.
$$
In light of Lemma~\ref{lem_equ} and Equation (\ref{sys}), one obtains $\lim_{t\to+\infty}\dot{\varphi}_i(t)=0$, $i\in I_N$. Considering that $\frac{\partial\nabla_{{p}_i}J(p_i^{*})}{\partial\varphi_{i+1}}$ and $\frac{\partial\nabla_{{p}_i}J(p_i^{*})}{\partial\varphi_{i}}$ are bounded, this allows to get
$$
\lim_{t\to\infty}\mathbf{H}_{J,p_i}(p_i^{*})\cdot\dot{p}_i^{*}(t)={0}.
$$
In addition, according to the implicit function theorem, if $\mathbf{H}_{J,p_i}(p_i^{*},\varphi_i)$ has full rank, there exists a unique continuously differentiable function $p_i^{*}(\varphi_i)$ such that $\nabla_{{p}_i}J(p_i^{*}(\varphi_i),\varphi_i)={0}$. Then it follows from Lemma~\ref{lem_phi} that $\lim_{t\to\infty}p_i^{*}(t)=\lim_{\varphi_i\to\varphi_i^{*}}p_i^{*}(\varphi_i)=p_{i,\infty}^{*}$. Therefore, if $\mathbf{H}_{J,p_i}(p_{i,\infty}^{*})$ has the full rank (i.e., $\text{rank}[\mathbf{H}_{J,p_i}(p_{i,\infty}^{*})]=2$), one gets $\lim_{t\to\infty}\dot{p}^{*}_i(t)=0$. \\
For Claim 2), by substituting (\ref{dp}) into (\ref{u_star}), one obtains
\[\dot{p}_i+\kappa_p{p_i} = {\kappa_p}{p_i^{*}(t)}, \quad i\in I_N \]
Solving the above differential equation yields
\[{p_i}(t) = {p_i}\left( 0 \right){e^{ - {\kappa_p}t}} + {e^{ - {\kappa_p}t}}\int_{0}^{t} {{\kappa_p}p_i^{*}(\tau){e^{{\kappa_p}\tau}}d\tau}.\]
Applying partial integration to the integral term of the above equation allows to get
\[{p_i}(t) = \left({p_i}\left( 0 \right)- p_i^{*}(0)\right){e^{-{\kappa _p}t}}+p_i^{*}(t)-{e^{-{\kappa_p}t}}\int_{0}^{t}{{e^{{\kappa_p}\tau}}d{p_i^{*}(\tau)}}.\]
By using the L'Hospital's rule for the last term of the above equation, one gets
\begin{equation*}
\begin{split}
\mathop {\lim }\limits_{t \to \infty } {e^{-{\kappa _p}t}}\int_{0}^{t}{{e^{{\kappa_p}\tau}}d{p_i^{*}(\tau)}}  & = \mathop {\lim }\limits_{t \to \infty } \frac{{\int_{0}^{t} {{e^{{\kappa _p}\tau}}d{p_i^{*}(\tau)}} }}{{{e^{{\kappa_p}t}}}} \\
& = \mathop {\lim }\limits_{t \to \infty } \frac{{{e^{{\kappa _p}t}}{{\dot p}_i^{*}}}(t)}{{{\kappa_p}{e^{{\kappa_p}t}}}} \\
& = \mathop {\lim }\limits_{t \to \infty } \frac{{{{\dot p}_i^{*}}}(t)}{{{\kappa_p}}}.
\end{split}
\end{equation*}
According to Lemma \ref{dce}, this leads to
$$
\mathop {\lim }\limits_{t \to \infty } {e^{-{\kappa _p}t}}\int_{0}^{t}{{e^{{\kappa _p}\tau}}d{p_i^{*}(\tau)}}=0.
$$
Therefore, it follows that
\begin{equation*}
\begin{split}
\lim_{t\to\infty}{p_i}(t)&=\lim_{t\to\infty}\left({p_i}\left( 0 \right)- p_i^{*}(0)\right){e^{-{\kappa _p}t}}+\lim_{t\to\infty}p_i^{*}(t)-\lim_{t\to\infty}{e^{-{\kappa_p}t}}\int_{0}^{t}{{e^{{\kappa _p}\tau}}d{p_i^{*}(\tau)}}\\
&=\lim_{t\to\infty}p_i^{*}(t)= p_{i,\infty}^{*}, \quad i \in I_N,
\end{split}
\end{equation*}
which indicates $\lim_{t\to\infty}\left\|\nabla_{{p}_i}J(p_i(t))\right\|=\left\|\nabla_{{p}_i}J(p_{i,\infty}^{*})\right\|=0$. \\
For Claim 3), it follows from Lemma~\ref{lem_equ}, Lemma~\ref{lem_comp} and Lemma~\ref{lem_exp} that
$$
V(\mathbf{\varphi})=\frac{1}{2}\sum\limits_{i=1}^N\left(m_i-\bar{m}\right)^2\leq \frac{N\sum_{i=1}^{N}(m_i-m_{i-1})^2}{2\lambda_{\min}(S)}\leq \frac{N^2c^2_1}{2\lambda_{\min}(S)}\cdot e^{-2c_2t}
$$
with the positive constants $c_1$ and $c_2$, which implies that $V(\mathbf{\varphi})$ exponentially converges to zero as time goes to the infinity. Thus,
each sub-region exponentially gains equal workload as time approaches the infinity. This completes the proof.
\end{proof}

\begin{remark}
Compared with the Voronoi-like partitions, the rotational bar partition is more suitable for irregular non-convex regions with load
balancing, and it is not necessary to spend time in constructing the Voronoi diagrams and updating multiple boundaries of polygons.
\end{remark}

\begin{lemma}\label{lem_fun}
There exists a differentiable function $\xi:[0,2\pi]\rightarrow[0,2\pi]$ such that
$$
\int_{\varphi}^{\xi(\varphi)}\omega(\theta)d\theta=\bar{m}, \quad \forall \varphi\in[0,2\pi]
$$
with $\omega(\theta)=\int_{r_{in}(\theta)}^{r_{out}(\theta)}\rho(r,\theta)rdr$.
\end{lemma}

\begin{proof}
First of all, one proves the existence of function $\xi:[0,2\pi]\rightarrow[0,2\pi]$ with the aforementioned equality constraint. Construct a function $G(\xi)=\int_{\varphi}^{\xi}\omega(\theta)d\theta-\bar{m}$, $\varphi\in[0,2\pi]$ and one gets $G(\varphi)=-\bar{m}<0$ and $G(\varphi+2\pi)=(N-1)\bar{m}>0$. It follows from the intermediate value theorem  that there exists a value $\xi_{\varphi}\in [\varphi,\varphi+2\pi]=[0,2\pi]$ such that $G(\xi_{\varphi})=0$, which is equivalent to $\int_{\varphi}^{\xi_{\varphi}}\omega(\theta)d\theta=\bar{m}$. This implies that there exists a function $\xi:[0,2\pi]\rightarrow[0,2\pi]$ satisfying the equality constraint. Next, one proves the continuity of function $\xi:[0,2\pi]\rightarrow[0,2\pi]$.  For any given $\varphi_0\in [0.2\pi]$, one obtains
$$
\int_{\varphi_0}^{\xi(\varphi_0)}\omega(\theta)d\theta=\int_{\varphi_0}^{\varphi}\omega(\theta)d\theta+\int_{\varphi}^{\xi(\varphi_0)}\omega(\theta)d\theta=\int_{\varphi}^{\xi(\varphi)}\omega(\theta)d\theta=\bar{m},
$$
which leads to
\begin{equation*}
\begin{split}
\int_{\varphi_0}^{\varphi}\omega(\theta)d\theta=\int_{\varphi}^{\xi(\varphi)}\omega(\theta)d\theta-\int_{\varphi}^{\xi(\varphi_0)}\omega(\theta)d\theta=\int_{\xi(\varphi_0)}^{\xi(\varphi)}\omega(\theta)d\theta
\end{split}
\end{equation*}
Thus, one can get
$$
\underline{\omega}|\xi(\varphi)-\xi(\varphi_0)|\leq\left|\int_{\xi(\varphi_0)}^{\xi(\varphi)}\omega(\theta)d\theta\right|=\left|\int_{\varphi_0}^{\varphi}\omega(\theta)d\theta\right|\leq\bar{\omega}|\varphi-\varphi_0|,
$$
which implies
\begin{equation}\label{nest}
|\xi(\varphi)-\xi(\varphi_0)|\leq\frac{\bar{\omega}}{\underline{\omega}}|\varphi-\varphi_0|.
\end{equation}
Then it concludes that for all $\epsilon>0$, there exists $\delta=\frac{\underline{\omega}\epsilon}{\bar{\omega}}>0$ such that $|\xi(\varphi)-\xi(\varphi_0)|\leq\epsilon$ whenever $|\varphi-\varphi_0|\leq\delta$. According to $\delta$-$\epsilon$ definition of limit, one can get $\lim_{\varphi\to\varphi_0}\xi(\varphi)=\xi(\varphi_0)$, which indicates the continuity of function $\xi:[0,2\pi]\rightarrow[0,2\pi]$. Finally, one proves the differentiability of function $\xi:[0,2\pi]\rightarrow[0,2\pi]$. For any given $\varphi_0\in [0.2\pi]$, by setting $\varphi=\varphi_0+h$, one can get
$$
\int_{\xi(\varphi_0)}^{\xi(\varphi_0+h)}\omega(\theta)d\theta=\int_{\varphi_0}^{\varphi_0+h}\omega(\theta)d\theta,
$$
which leads to $[\xi(\varphi_0+h)-\xi(\varphi_0)]\cdot\omega(\xi_h)=h\cdot\omega(\varphi_h)$ with $\xi_h\in[\xi(\varphi_0),\xi(\varphi_0+h)]$ and $\varphi_h\in[\varphi_0,\varphi_0+h]$ according to mean value theorems for definite integrals. Thus, it follows from the  continuity of functions $\xi(\varphi)$ and $\omega(\theta)$ that
$$
\lim_{h\to0}\frac{\xi(\varphi_0+h)-\xi(\varphi_0)}{h}=\lim_{h\to0}\frac{\omega(\varphi_h)}{\omega(\xi_h)}=\frac{\omega(\varphi_0)}{\omega(\xi(\varphi_0))}.
$$
The proof is thus completed.
\end{proof}

\begin{prop}\label{prop_opt}
There exist optimal solutions to Problem (\ref{min_cost}).
\end{prop}

\begin{proof}\
Note that there exists a continuous function $\xi:[0,2\pi]\rightarrow[0,2\pi]$ such that $\int_{\varphi}^{\xi(\varphi)}\omega(\theta)d\theta=\bar{m}$, $\forall \varphi\in[0,2\pi]$. Thus, one has $\varphi_i=\xi(\varphi_{i-1})$, $i\in I_N$ when equitable workload partition is completed. Then the cost function (\ref{cost}) can be rewritten as
$$
J(\mathbf{\varphi},\mathbf{p})=\sum_{i=1}^N\int_{\varphi_i}^{\xi(\varphi_i)}\int_{r_{in}(\theta)}^{r_{out}(\theta)}f({p}_i,r,\theta)\rho(r,\theta)rdrd\theta
=\sum_{i=1}^N\int_{\xi^{i-1}(\varphi_1)}^{\xi^i(\varphi_1)}\int_{r_{in}(\theta)}^{r_{out}(\theta)}f({p}_i,r,\theta)\rho(r,\theta)rdrd\theta,
$$
where $\xi^i$ denotes the function $\xi\circ\xi\cdot\cdot\cdot\xi$ with $i$ compositions, $\varphi_1\in[\varphi^{*}_1,\varphi^{*}_2]$ and ${p}_i\in E_i$, $\forall i\in I_N$.
Since $\xi(\varphi_1)$ is a continuous function, $\xi^i(\varphi_1)$ is continuous as well. Then $J(\mathbf{\varphi},\mathbf{p})=J(\mathbf{\varphi_1},\mathbf{p})$ is continuous with respect to $\varphi_1$ and ${p}_i$, and $[\varphi^{*}_1,\varphi^{*}_2]$ and $E_i$, $i\in I_N$ are all closed and bounded domains (i.e. compact set). This implies the existence of optimal solutions to the problem (\ref{min_cost}) according to the extreme value theorem.
\end{proof}

\begin{figure}[t!]
\centering
\includegraphics[width=12cm]{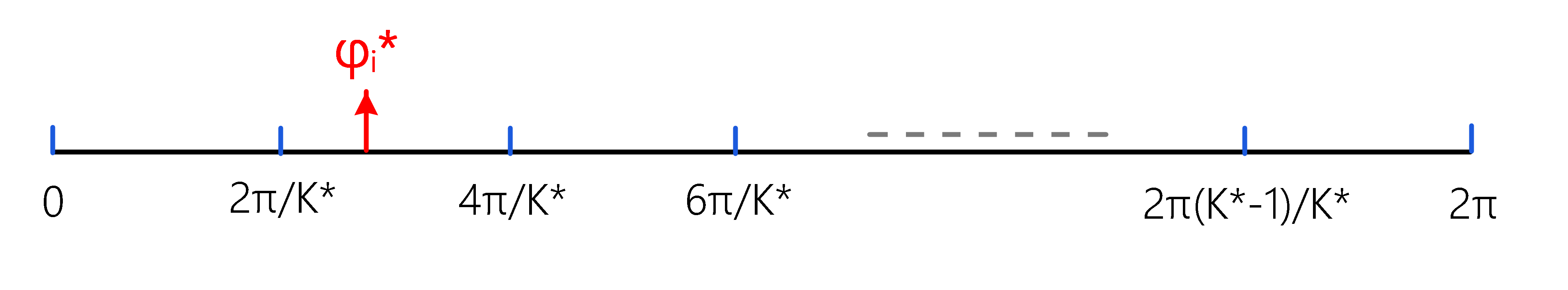}
\caption{Equal partition of phase angles on $[0,2\pi]$.}\label{angle}
\end{figure}


\begin{algorithm}[t!]
\caption{\label{tab:cam} Circular Search-and-Monitoring Algorithm}
\hspace*{0.02in} {\bf Initialize:}  $K^{*}$, $T_{\epsilon}$, $p_i$, $\varphi_i$, $\kappa_{\varphi}$, $\kappa_p$ \\
\hspace*{0.02in} For $i \in I_N$, the $i$-th agent performs as follows
\begin{algorithmic}[1]
\For{$k=1:K^{*}$}
\State Set $t=0$
\While{$t<T_{\epsilon}$}
    \If{$i=\min_{l_k\in\Lambda_k}l_k$}
       \State Set $\varphi_i=2\pi(k-1)/K^{*}$
       \State Update ${p}_i$ with (\ref{dp}) and (\ref{u_star})
    \Else
       \State Update $\varphi_i$ with (\ref{sys})
       \State Update ${p}_i$ with (\ref{dp}) and (\ref{u_star})
    \EndIf \State {\bf end~if}
\EndWhile \State {\bf end~while}
\State Save ${C}^k_i=(\varphi^k_i,{p}^k_i,J^k_{E_i})$ and $\Omega^k_i=\{J^k_{E_i}\}$
\State Receive $\Omega^k_{i-1}$ from $(i-1)$-th agent
\While{$|\Omega^k_i|\neq|\Omega^k_i\bigcup\Omega^k_{i-1}|$}
\State Update $\Omega^k_i=\Omega^k_i\bigcup\Omega^k_{i-1}$
\State Send $\Omega^k_i$ to $(i+1)$-th agent
\State Receive $\Omega^k_{i-1}$ from $(i-1)$-th agent
\EndWhile \State {\bf end~while}
\State Compute $J_i^k=\sum_{\Omega^k_i}J^k_{E_i}$
\EndFor \State {\bf end~for}
\State Select $k^{*}=\arg\min_{1\leq k\leq K^{*}}J_i^k$
\State Finalize $\varphi_i$ and ${p}_i$ with ${C}^{k^{*}}_i$
\end{algorithmic}
\end{algorithm}
Note that ${p}^{*}_i$ is dependent on $\varphi_{i-1}$ and $\varphi_i$, and it is ultimately dependent on $\varphi_{1}$ due to $\varphi_i=\xi(\varphi_{i-1})=\xi^{i-1}(\varphi_{1})$. As a result, Problem (\ref{min_cost}) can be equivalently converted into
$\min_{\varphi_1\in[0,2\pi]} J(\varphi_1,\mathbf{p}^{*}(\varphi_1))$. Define $K^{*}=\arg\inf_{k\in Z^{+}}\{k>0|{2\pi}/{k}\leq\epsilon_p\}$, then the interval $[0,2\pi]$ can be evenly divided into $K^{*}$ sub-intervals, i.e. $[0,2\pi]=\bigcup_{k=1}^{K^{*}}[2\pi(k-1)/K^{*},2\pi k/K^{*}]$, as shown in Fig.~\ref{angle}. The index set $\Lambda_k=\{l_k\in Z^{+}| l_k=\arg\inf_{i\in I_N}|\varphi_i-(k-1)2\pi/K^{*}|\}$ is introduced to identify $\varphi_i$ that is closest to $2\pi(k-1)/K^{*}$. In what follows, a distributed search-and-monitoring scheme is developed in \textbf{Algorithm}~\ref{tab:cam}, which enables multi-agent systems attaining optimal solutions to Problem (\ref{min_cost}) with any specified accuracy. First of all, some key parameters are initialized before activating the algorithm. Then the $i$-th agent determines if it is closet to the phase angle $2\pi(k-1)/K^{*}$. If this is true, the partition bar of the $i$-th agent shifts to $2\pi(k-1)/K^{*}$ and keeps unchanged. Meanwhile, the position of the $i$-th agent ${p}_i$ is updated with (\ref{dp}) and (\ref{u_star}). If $\varphi_i$  is not closet to $2\pi(k-1)/K^{*}$, it is updated with (\ref{sys}) and ${p}_i$ is still updated with (\ref{dp}) and (\ref{u_star}). The above process lasts for a time period of $T_{\epsilon}$. Then the $i$-th agent saves the results on $\varphi^k_i$, ${p}^k_i$, $J^k_{E_i}$ in the set ${C}^k_i$, while creating a new set $\Omega^k_i=\{J^k_{E_i}\}$. After that, the $i$-th agent communicates with $(i-1)$-th agent to obtain $\Omega^k_{i-1}$ and checks whether $|\Omega^k_i|\neq|\Omega^k_i\bigcup\Omega^k_{i-1}|$ holds. If the above condition is satisfied, $\Omega^k_i$ is updated with $\Omega^k_i=\Omega^k_i\bigcup\Omega^k_{i-1}$ before sending to the $(i-1)$-th agent. Subsequently, the $i$-th agent continues to receive $\Omega^k_{i-1}$ from $(i-1)$-th agent. Once the set $\Omega^k_i$ is no longer updated, $J_i^k$ is computed with $J_i^k=\sum_{\Omega^k_i}J^k_{E_i}$. The above loop lasts for $K^{*}$ times, then the $(i-1)$-th agent selects the index $k^{*}=\arg\min_{1\leq k\leq K^{*}}J_i^k$ in order to finalize $\varphi_i$ and ${p}_i$ with ${C}^{k^{*}}_i$.

\begin{theorem}
\textbf{Algorithm}~\ref{tab:cam} ensures multi-agent systems converge to optimal solutions to Problem (\ref{min_cost}) with arbitrary small tolerance.
\end{theorem}

\begin{proof}
Let $J(\varphi^{*},\mathbf{p}^{*})$ and $J(\varphi^{\epsilon},\mathbf{p}^{\epsilon})$ denote the optimal solution to Problem (\ref{min_cost}) and actual solution with the proposed algorithm, respectively. Then their error can be estimated by
\begin{equation*}
\begin{split}
|J(\varphi^{*},\mathbf{p}^{*})-J(\varphi^{\epsilon},\mathbf{p}^{\epsilon})|&=|J(\varphi^{*},\mathbf{p}^{*})-J(\varphi^k,\mathbf{p}^{*})+J(\varphi^k,\mathbf{p}^{*})-J(\varphi^k,\mathbf{p}^{k})+J(\varphi^k,\mathbf{p}^{k})-J(\varphi^{\epsilon},\mathbf{p}^{\epsilon})|\\
&\leq |J(\varphi^{*},\mathbf{p}^{*})-J(\varphi^k,\mathbf{p}^{*})|+|J(\varphi^k,\mathbf{p}^{*})-J(\varphi^{k},\mathbf{p}^k)|+|J(\varphi^k,\mathbf{p}^{k})-J(\varphi^{\epsilon},\mathbf{p}^{\epsilon})|.
\end{split}
\end{equation*}
The above three terms can be estimated one by one as follows. First of all, define
$$
\omega_{f}(\theta,{p}_i^{*})=\int_{r_{in}(\theta)}^{r_{out}(\theta)}f({p}^{*}_i,r,\theta)\rho(r,\theta)rdr,
$$
and one can get
\begin{equation*}
\begin{split}
|J(\varphi^{*},\mathbf{p}^{*})-J(\varphi^k,\mathbf{p}^{*})|&=\left|\sum_{i=1}^N\left[\int_{\varphi^{*}_i}^{\xi(\varphi^{*}_i)}\omega_{f}(\theta,{p}_i^{*})d\theta
-\int_{\varphi^k_i}^{\xi(\varphi^k_i)}\omega_{f}(\theta,{p}_i^{*})d\theta\right]\right|\\
&=\left|\sum_{i=1}^N\left[\int_{\varphi^{*}_i}^{\varphi^{k}_i}\omega_{f}(\theta,{p}_i^{*})d\theta
+\int_{\xi(\varphi^k_i)}^{\xi(\varphi^{*}_i)}\omega_{f}(\theta,{p}_i^{*})d\theta\right]\right|\\
&\leq \sum_{i=1}^N\left|\int_{\varphi^{*}_i}^{\varphi^{k}_i}\omega_{f}(\theta,{p}_i^{*})d\theta\right|
+\sum_{i=1}^N\left|\int_{\xi(\varphi^k_i)}^{\xi(\varphi^{*}_i)}\omega_{f}(\theta,{p}_i^{*})d\theta\right|\\
&\leq \bar{\omega}_{f}\sum_{i=1}^N |\varphi^{*}_i-\varphi^{k}_i|+\bar{\omega}_{f}\sum_{i=1}^N |\xi(\varphi^{*}_i)-\xi(\varphi^k_i)|.
\end{split}
\end{equation*}
Note that $|\xi(\varphi^{*}_i)-\xi(\varphi^k_i)|\leq \frac{\bar{\omega}}{\underline{\omega}}|\varphi^{*}_i-\varphi^{k}_i|$ according to Inequality (\ref{nest})
in Lemma~\ref{lem_fun}, and this yields
\begin{equation*}
\begin{split}
|J(\varphi^{*},\mathbf{p}^{*})-J(\varphi^k,\mathbf{p}^{*})|&\leq \bar{\omega}_{f}\sum_{i=1}^N |\varphi^{*}_i-\varphi^{k}_i|+\bar{\omega}_{f}\sum_{i=1}^N |\xi(\varphi^{*}_i)-\xi(\varphi^k_i)|\\
&\leq \bar{\omega}_{f}\left(1+\frac{\bar{\omega}}{\underline{\omega}}\right)\sum_{i=1}^N|\varphi^{*}_i-\varphi^{k}_i|.
\end{split}
\end{equation*}
It follows from $\varphi^{*}_i-\varphi^{k}_i=\xi(\varphi^{*}_{i-1})-\xi(\varphi^{k}_{i-1})$ that
$$
\sum_{i=1}^N|\varphi^{*}_i-\varphi^{k}_i|\leq\epsilon_p\sum_{i=1}^N\left(\frac{\bar{\omega}}{\underline{\omega}}\right)^{i-1}.
$$
Thus, one can get
$$
|J(\varphi^{*},\mathbf{p}^{*})-J(\varphi^k,\mathbf{p}^{*})|\leq\epsilon_p\bar{\omega}_{f}\left(1+\frac{\bar{\omega}}{\underline{\omega}}\right)\sum_{i=1}^N\left(\frac{\bar{\omega}}{\underline{\omega}}\right)^{i-1}.
$$
Secondly, one has
\begin{equation*}
\begin{split}
|J(\varphi^k,\mathbf{p}^{*})-J(\varphi^k,\mathbf{p}^k)|&=\left|\sum\limits_{i = 1}^N \int_{E_i(\mathbf{\varphi}^k)}f({p}^{*}_i,q)\rho(q)dq
-\sum\limits_{i = 1}^N\int_{E_i(\mathbf{\varphi}^k)}f({p}^{k}_i,q)\rho(q)dq\right| \\
&= \left|\sum\limits_{i = 1}^N\int_{E_i(\mathbf{\varphi}^k)}f({p}^{*}_i,q)\rho(q)-f({p}^{k}_i,q)\rho(q)dq \right| \\
&\leq \sum\limits_{i = 1}^N\int_{E_i(\mathbf{\varphi}^k)}|f({p}^{*}_i,q)-f({p}^{k}_i,q)|\rho(q)dq.
\end{split}
\end{equation*}
Since $|f({p}^{*}_i,q)-f({p}^{k}_i,q)|\leq\sup_{\alpha\in[0,1],q\in E_i(\varphi^k)}\|\nabla_{{p}_i}f(\alpha{p}^{*}_i+(1-\alpha{p}^{k}_i),q)\|\cdot\|{p}^{*}_i-{p}^{k}_i\|$, one can obtain
\begin{equation*}
\begin{split}
|J(\varphi^k,\mathbf{p}^{*})-J(\varphi^k,\mathbf{p}^k)|&\leq \sum\limits_{i = 1}^N\int_{E_i(\mathbf{\varphi}^k)}\rho(q)dq \sup_{\alpha\in[0,1],q\in E_i(\varphi^k)}\|\nabla_{{p}_i}f(\alpha{p}^{*}_i+(1-\alpha){p}^{k}_i,q)\|\cdot\|{p}^{*}_i-{p}^{k}_i\| \\
&=\bar{m} \sum_{i=1}^N L_i \|{p}^{*}_i-{p}^{k}_i\| \\
\end{split}
\end{equation*}
with $L_i=\sup_{\alpha\in[0,1],q\in E_i(\varphi^k)}\|\nabla_{{p}_i}f(\alpha{p}^{*}_i+(1-\alpha{p}^{k}_i),q)\|$.
Moreover, if $\mathbf{H}_{J,p_i}({p}_i^{*})$ has full rank, the term $\|{p}^{*}_i-{p}^{k}_i\|$ can be estimated by
$$
\|{p}^{*}_i-{p}^{k}_i\|\leq\sup_{\varphi_i\in[0,2\pi]}\left\|\frac{\partial{p}^{*}_i}{\partial\varphi_i}\right\|\cdot|\varphi^{*}_i-\varphi^k_i|,
$$
which leads to
\begin{equation*}
\begin{split}
|J(\varphi^k,\mathbf{p}^{*})-J(\varphi^k,\mathbf{p}^k)|&\leq\bar{m}\sum_{i=1}^N L_i \|{p}^{*}_i-{p}^{k}_i\| \\
&\leq \bar{m}\sum_{i=1}^N L_i\sup_{\varphi_i\in[0,2\pi]}\left\|\frac{\partial{p}^{*}_i}{\partial\varphi_i}\right\|\cdot|\varphi^{*}_i-\varphi^k_i| \\
&\leq \bar{m}L\sum_{i=1}^N|\varphi^{*}_i-\varphi^k_i|\leq \epsilon_p\bar{m}L\sum_{i=1}^N\left(\frac{\bar{\omega}}{\underline{\omega}}\right)^{i-1}
\end{split}
\end{equation*}
with $L=\sup_{\varphi_i\in[0,2\pi],i\in I_N}\left\|\frac{\partial{p}^{*}_i}{\partial\varphi_i}\right\|L_i$. Finally, considering that $\lim_{t\to\infty}J(\varphi^{\epsilon},\mathbf{p}^{\epsilon})=J(\varphi^k,\mathbf{p}^{k})$, for any $\epsilon>0$, there exists $T_{\epsilon}>0$ such that
$|J(\varphi^k,\mathbf{p}^{k})-J(\varphi^{\epsilon},\mathbf{p}^{\epsilon})|<\epsilon$ whenever $t>T_{\epsilon}$. Therefore, the error $|J(\varphi^{*},\mathbf{p}^{*})-J(\varphi^{\epsilon},\mathbf{p}^{\epsilon})|$ is bounded by
\begin{equation*}
\begin{split}
|J(\varphi^{*},\mathbf{p}^{*})-J(\varphi^{\epsilon},\mathbf{p}^{\epsilon})|&\leq\epsilon_p\bar{\omega}_{f}\left(1+\frac{\bar{\omega}}{\underline{\omega}}\right)\sum_{i=1}^N\left(\frac{\bar{\omega}}{\underline{\omega}}\right)^{i-1}
+\epsilon_p\bar{m}L\sum_{i=1}^N\left(\frac{\bar{\omega}}{\underline{\omega}}\right)^{i-1}+\epsilon, \quad t\geq T_{\epsilon}.
\end{split}
\end{equation*}
Note that the terms $\epsilon_p$ and $\epsilon$ can be arbitrary small by enlarging $K^{*}$ and $T_{\epsilon}$, respectively. This completes the proof.
\end{proof}

\begin{remark}
According to the implicit function theorem, if $\mathbf{H}_{J,p_i}(p_i^{*},\varphi_i)$ has full rank, there exists a unique continuously differentiable function $p_i^{*}(\varphi_i)$ such that $\nabla_{{p}_i}J(p_i^{*}(\varphi_i),\varphi_i)={0}$. Moreover, one can get
$$
\frac{\partial{p}^{*}_i}{\partial\varphi_i}=-H_{J,p_i}(p^{*}_i)^{-1}\frac{\partial\nabla_{{p}_i}J(p_i^{*})}{\partial\varphi_i}.
$$
In addition, by fixing one phase angle of partition bar, equitable workload partition can still be achieved with a different rate of exponential convergence~\cite{auto13}.
\end{remark}

The above theoretical results can be employed to solve the classic location optimization problem~\cite{cor04}, where the function $f({q}_i,p)=\|{p}_i-q\|^2$ quantifies the cost of the $i$-th agent to service the event occurred at the point $q$. According to (\ref{cost}),  this allows to get the following performance index.
$$
J(\mathbf{\varphi},\mathbf{p})= \sum_{i=1}^N \int_{E_i(\mathbf{\varphi})}\|{p}_i-q\|^2\rho(q)dq.
$$
For the above location optimization problem, one can obtain two corollaries as follows.
\begin{coro}\label{p2cei_lem}
Dynamical system (\ref{dp}) with control input (\ref{u_star}) ensures that each agent can move to the centroid of sub-region, $i.e.$
$\lim_{t\to\infty}{{p}_i}(t)={c}^{*}_{E_i}$, $\forall i \in I_N$.
\end{coro}

\begin{proof}
It follows from $f({p}_i,q)=\|{p}_i-q\|^2$ and the cost function (\ref{cost}) that
\begin{equation*}
\begin{aligned}
\nabla_{{p}_i}J &=2 {\int_{{E_i}} {\rho \left(q\right)}}\left({{p_i}-q}\right)^Tdq \\
 &=2 {{p_i}\int_{{E_i}} {\rho \left( q \right)} dq - 2\int_{{E_i}} {q\rho \left( q \right)dq} } \\
 &=2 m_i\left( {{p_i} - {c_{E_i}}}\right), \quad \forall i \in I_N,
\end{aligned}
\end{equation*}
where
\begin{equation}\label{cei}
{c}_{{E_i}}=\frac{{\int_{{E_i}} {q\rho \left( q \right)dq} }}{{\int_{{E_i}} {\rho \left( q \right)} dq}}
\end{equation}
represents the centroid of sub-region $E_i$. According to Claim 2) in Theorem~\ref{theo1}, one has $\lim_{t\to\infty}\|\nabla_{{p}_i}J\|=0$.
Considering that $\|\nabla_{{p}_i}J\|=2 m_i\|p_i-c_{E_i}\|$ and $\lim_{t\to\infty}m_i(t)=\bar{m}>0$ according to Lemma~\ref{lem_equ}, it follows that
$$
\lim_{t\to\infty}\|\nabla_{{p}_i}J\|=2\lim_{t\to\infty}m_i(t)\cdot\lim_{t\to\infty}\|p_i-c_{E_i}\|=2\bar{m}\cdot\lim_{t\to\infty}\|p_i-c_{E_i}\|=0,
$$
which leads to $\lim_{t\to\infty}\|p_i(t)-c_{E_i}(t)\|=0$. Moreover, note that ${c}_{E_i}=(c^x_{E_i},c^y_{E_i})$ is a continuous function of $\varphi_i$,
$i\in I_N$ and it can be computed as
\begin{equation*}
\begin{aligned}
c_{{E_i}}^x=\left\{
	\begin{aligned}
	&{\frac{{\int_{{\varphi _i}}^{2\pi } {{\omega _x}\left( \theta  \right)d\theta }  + \int_0^{{\varphi _{i + 1}}} {{\omega _x}\left( \theta  \right)d\theta } }}{{{m_i}}}},\quad & {{\varphi _{i + 1}} < {\varphi _i}}\\
	&{\frac{{\int_{{\varphi _i}}^{{\varphi _{i + 1}}} {{\omega _x}\left( \theta  \right)d\theta } }}{{{m_i}}}.} \quad & otherwise\\
	\end{aligned}
	\right.\\
c_{{E_i}}^y=\left\{
	\begin{aligned}
	&{\frac{{\int_{{\varphi _i}}^{2\pi } {{\omega _y}\left( \theta  \right)d\theta }  + \int_0^{{\varphi _{i + 1}}} {{\omega _y}\left( \theta  \right)d\theta } }}{{{m_i}}}},\quad & {{\varphi _{i + 1}} < {\varphi _i}}\\
	&{\frac{{\int_{{\varphi _i}}^{{\varphi _{i + 1}}} {{\omega _y}\left( \theta  \right)d\theta } }}{{{m_i}}}.} \quad & otherwise\\
	\end{aligned}
	\right.
\end{aligned}
\end{equation*}
with ${\omega _x}\left(\theta\right) = \int_{{r_{in}}\left( \theta  \right)}^{{r_{out}}\left( \theta  \right)} {r\cos \left(\theta\right)\rho \left( {r,\theta } \right)rdr}$
and ${\omega _y}\left( \theta\right) = \int_{{r_{in}}\left(\theta\right)}^{{r_{out}}\left(\theta  \right)}{r\sin \left(\theta\right)\rho \left( {r,\theta } \right)rdr}$.
According to Lemma~\ref{lem_phi}, one has
$$
\lim_{t\to\infty}c_{E_i}(t)=\lim_{\varphi_i\to\varphi^{*}_i}c_{E_i}(\varphi_i)=c^{*}_{E_i}, \quad i\in I_N
$$
Since $0\leq\|p_i(t)-c^{*}_{E_i}\|=\|p_i(t)-c_{E_i}(t)+c_{E_i}(t)-c^{*}_{E_i}\|\leq \|p_i(t)-c_{E_i}(t)\|+\|c_{E_i}(t)-c^{*}_{E_i}\|$, it allows to obtain
$$
0\leq \lim_{t\to\infty}\|p_i(t)-c^{*}_{E_i}\|\leq \lim_{t\to\infty}\|p_i(t)-c_{E_i}(t)\|+\lim_{t\to\infty}\|c_{E_i}(t)-c^{*}_{E_i}\|=0,
$$
which implies $\lim_{t\to\infty}p_i(t)=c^{*}_{E_i}$. The proof is thus completed.
\end{proof}

\begin{remark}
For fixed partitions, $\mathbf{\varphi}=(\varphi_1,\varphi_2,...,\varphi_N)^T$ is a constant vector and the cost function (\ref{cost}) is the continuous function of state variables $\mathbf{p}=({p}_1,{p}_2,...,{p}_N)^T$. Then it follows from
$\nabla_{{p}_i}J=2{m_i}\left({p}_i-c_{E_i}\right)=0$, $\forall i\in I_N$
that the set of centroid $\mathbf{p}^{*}=(c_{E_1},c_{E_2},...,c_{E_N})$ is the unique critical point of the cost function (\ref{cost}).
Moreover, its Hessian matrix satisfies
\begin{equation*}
\begin{split}
H[J(\mathbf{\varphi},\mathbf{p})]=\left[\frac{\partial^2J}{\partial p_i\partial p_j}\right] =2 \text{diag}(m_1,m_1,m_2,m_2,...,m_N,m_N)>0
\end{split}
\end{equation*}
at the critical point $\mathbf{p}^{*}=(c_{E_1},c_{E_2},...,c_{E_N})$. Thus, the cost function is a strictly convex function with respect to state variables $\mathbf{p}$. This indicates that the set of centroid ensures the minimum of cost function (\ref{cost}) with fixed partitions.
\end{remark}

\begin{coro}
Multi-agent dynamics (\ref{dp}) with (\ref{u_star}) is input-to-state stable.
\end{coro}

\begin{proof}
The time derivative of (\ref{cost}) with respect to the compound dynamics (\ref{sys})
and (\ref{dp}) is given by
\begin{equation*}
\begin{split}
\frac{dJ}{dt}&=\sum_{i=1}^{N}\frac{\partial J}{\partial {p}_i}\dot{p}_i+\sum_{i=1}^{N}\frac{\partial J}{\partial {\varphi}_i}\dot{\varphi}_i \\
&=2\sum_{i=1}^N m_i\left({{p_i}-{c_{E_i}}}\right)^T\dot{p}_i+\sum_{i=1}^{N}\dot{\varphi}_i\left[\eta(\varphi_i,p_{i-1})-\eta(\varphi_i,p_{i})\right] \\
&=-2\kappa_p\sum_{i = 1}^N m_i\|{{p_i}-{c_{E_i}}}\|^2+\sum_{i=1}^{N}\dot{\varphi}_i\left[ \eta(\varphi_i,p_{i-1})-\eta(\varphi_i,p_{i})\right] \\
\end{split}
\end{equation*}
with
\begin{equation*}
\begin{split}
\eta(\theta,s)=\int_{r_{in}(\theta)}^{r_{out}(\theta)}\rho(\theta,r)r^3dr+\|s\|^2\cdot\int_{r_{in}(\theta)}^{r_{out}(\theta)}\rho(\theta,r)rdr-2s^T\int_{r_{in}(\theta)}^{r_{out}(\theta)}\rho(\theta,r)
\left(
  \begin{array}{c}
    r\cos\theta \\
    r\sin\theta \\
  \end{array}
\right)
rdr.
\end{split}
\end{equation*}
In light of the parallel axis theorem, one has
$$
J=\sum_{i=1}^{N}\int_{E_i}\rho(q)\|q-c_{E_i}\|^2dq +\sum_{i=1}^{N} m_i\|{{p_i}-{c_{E_i}}}\|^2,
$$
which leads to
\begin{equation*}
\begin{split}
\frac{dJ}{dt}&=\sum_{i=1}^{N}\frac{\partial}{\partial\varphi_i}\int_{E_i}\rho(q)\|q-c_{E_i}\|^2dq\cdot\dot{\varphi}_i+\frac{d}{dt}\sum_{i=1}^{N} m_i\|{{p_i}-{c_{E_i}}}\|^2 \\
&=\sum_{i=1}^{N}\dot{\varphi}_i\left[ \eta(\varphi_i,c_{E_{i-1}})-\eta(\varphi_i,c_{E_i})\right]+\frac{d}{dt}\sum_{i=1}^{N} m_i\|{{p_i}-{c_{E_i}}}\|^2.
\end{split}
\end{equation*}
Therefore, one gets
\begin{equation*}
\begin{split}
\frac{d}{dt}\sum_{i=1}^{N} m_i\|{{p_i}-{c_{E_i}}}\|^2&=-2\kappa_p\sum_{i = 1}^N m_i\|{{p_i}-{c_{E_i}}}\|^2+\sum_{i=1}^{N}\dot{\varphi}_i\left[ \eta(\varphi_i,p_{i-1})-\eta(\varphi_i,p_{i})\right] \\
&-\sum_{i=1}^{N}\dot{\varphi}_i\left[\eta(\varphi_i,c_{E_{i-1}})-\eta(\varphi_i,c_{E_i})\right]
\end{split}
\end{equation*}
Let $H(t)=\sum_{i=1}^{N} m_i\|{{p_i}-{c_{E_i}}}\|^2$ and $e_{\eta}(\varphi_i,p_i,p_{i-1})=\eta(\varphi_i,p_{i-1})-\eta(\varphi_i,p_{i})-\eta(\varphi_i,c_{E_{i-1}})+\eta(\varphi_i,c_{E_i})$. Then it follows that
\begin{equation*}
\begin{split}
\dot{H}(t)&=-2\kappa_pH(t)+\sum_{i=1}^{N}\dot{\varphi}_i\cdot e_{\eta}(\varphi_i,p_i,p_{i-1}) \\
&\leq -2\kappa_pH(t)+\|\mathbf{\dot{\varphi}}\|\cdot\|\mathbf{e_{\eta}}(\varphi,p)\| \\
&\leq -2\kappa_pH(t)+\|\mathbf{\dot{\varphi}}(t)\|\cdot\sup_{\varphi\in[0,2\pi],p\in\Omega}\|\mathbf{e_{\eta}}(\varphi,p)\|
\end{split}
\end{equation*}
with $\mathbf{\dot{\varphi}}=(\dot{\varphi}_1,\dot{\varphi}_2,...,\dot{\varphi}_N)^T$ and
$\mathbf{e_{\eta}}(\varphi,p)=(e_{\eta}(\varphi_1,p_1,p_{N}),e_{\eta}(\varphi_2,p_2,p_{1}),...,e_{\eta}(\varphi_N,p_N,p_{N-1}))^T$. Solving the above differential inequality yields
\begin{equation*}
\begin{split}
H(t)&\leq H(t_0)e^{-2\kappa_p(t-t_0)}+\sup_{\varphi\in[0,2\pi],p\in\Omega}\|\mathbf{e_{\eta}}(\varphi,p)\|\cdot\int_{t_0}^{t} e^{2\kappa_p(\tau-t)}\|\mathbf{\dot{\varphi}}(\tau)\|d\tau \\
&\leq H(t_0)e^{-2\kappa_p(t-t_0)}+\sup_{\varphi\in[0,2\pi],p\in\Omega}\|\mathbf{e_{\eta}}(\varphi,p)\|\cdot\frac{\sup_{\tau\in[t_0,t]}\|\mathbf{\dot{\varphi}}(\tau)\|}{2\kappa_p},
\end{split}
\end{equation*}
which implies input-to-state stability of multi-agent dynamics according to the definition in~\cite{kha96}.
\end{proof}

\begin{remark}
If the reference point of coverage region is only available to a part of agents, a distributed estimator can be designed to obtain this reference point for each agent as follows:
$\dot{r}_i=\kappa_r(r_i-r_{i-1})$, $i\in I^r_N$ and $\dot{r}_i=0$, $i\notin I^r_N$, where $I^r_N$ denotes the set of agents with the reference point and $\kappa_r$ refers to a positive constant. Alternatively, it can be adopted as a special case of distributed observer of multi-agent system with a static leader in~\cite{hong08}.
\end{remark}

%
%

\begin{figure}[t!]
\centering
\includegraphics[width=16.5cm]{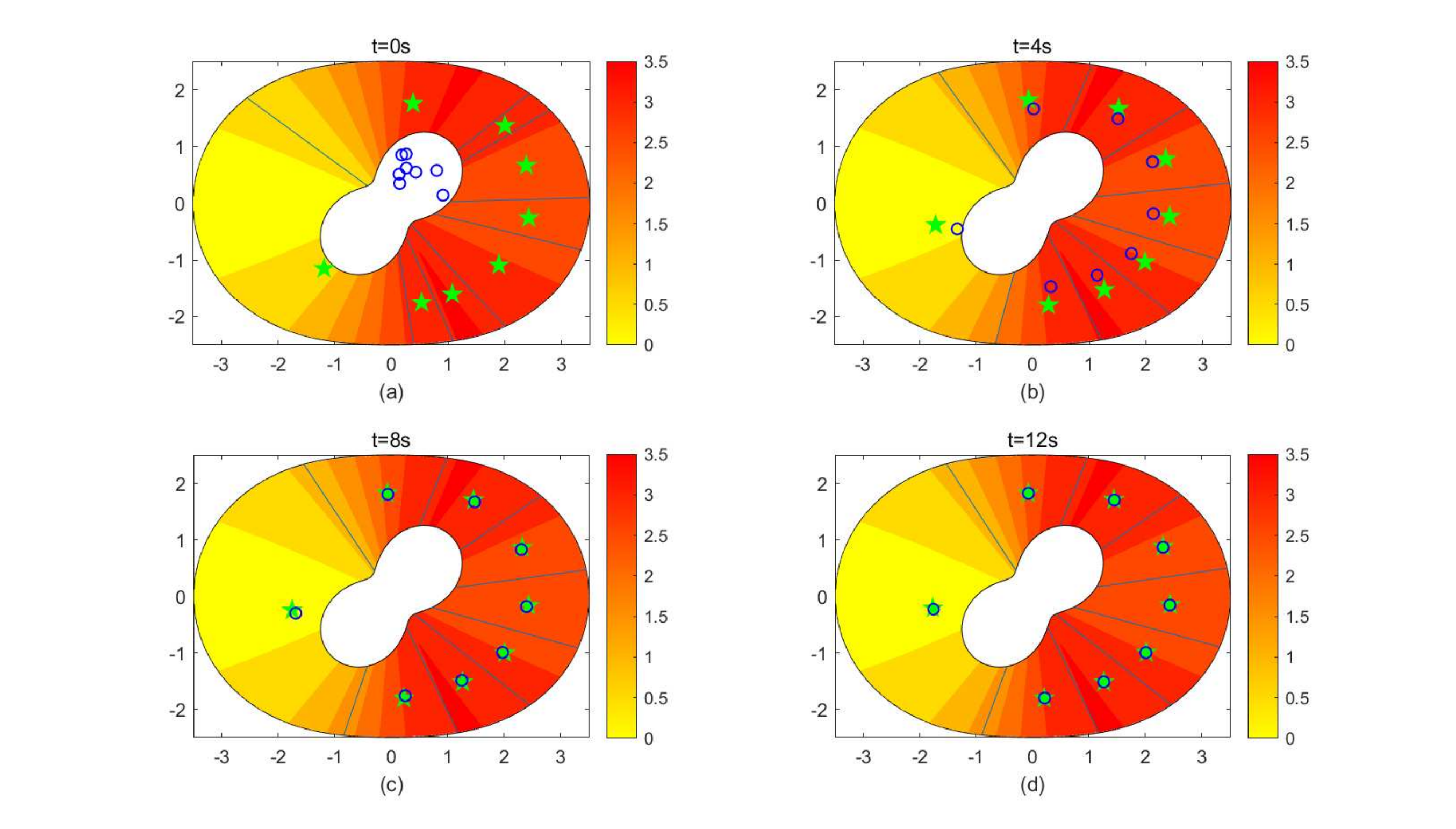}
\caption{Snapshots of simulation results on region partition, sub-region centroids and positions in multi-agent systems. Blue circles denote the mobile agents, and green stars refer to the sub-region centroids.}
\label{ac}
\end{figure}

\begin{figure}[t!]
\centering
\includegraphics[width=14cm]{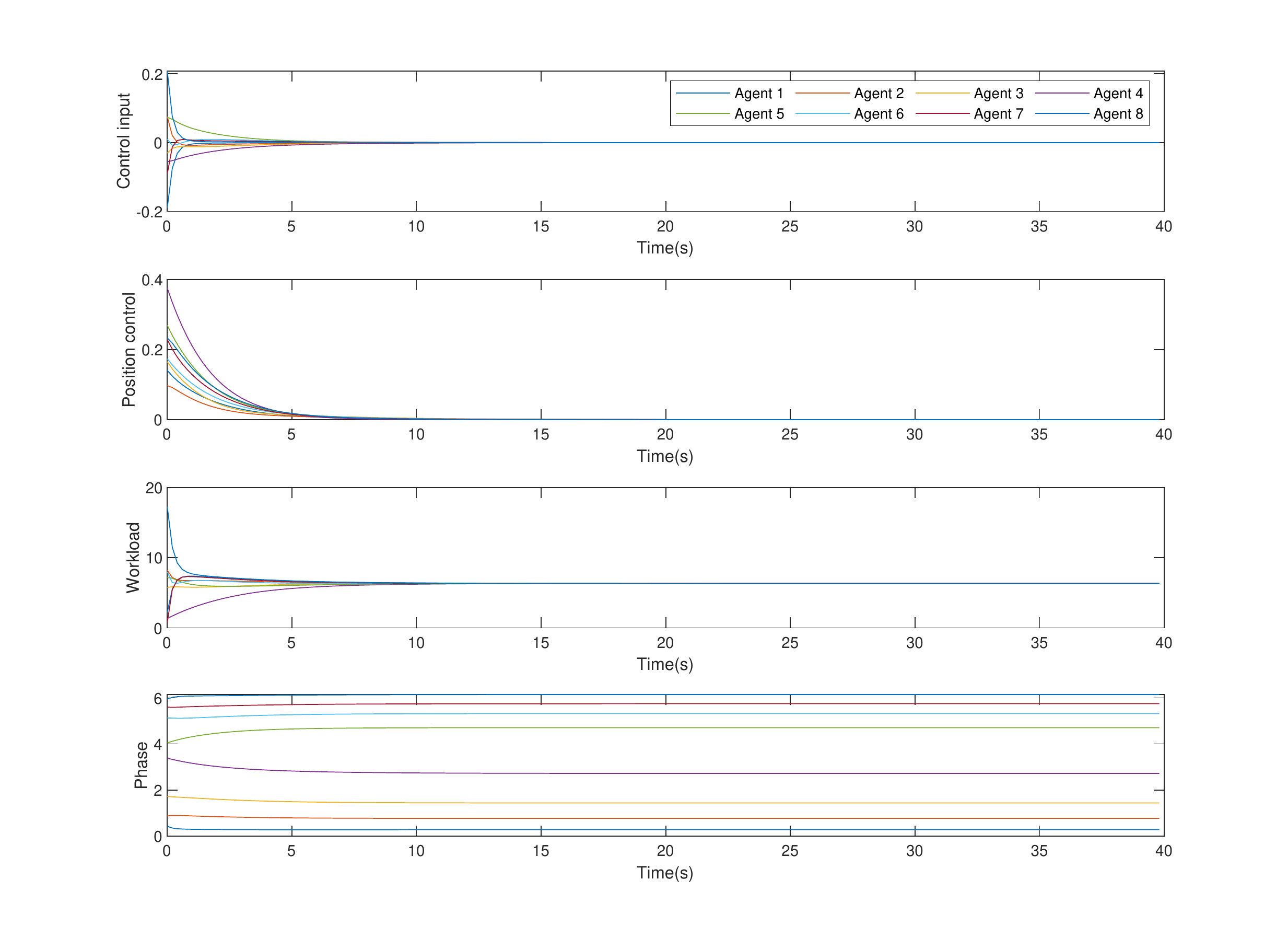}
\caption{Time evolution of physical variables in multi-agent systems.}
\label{workphase}
\end{figure}

\section{Case Studies}\label{sec:cas}
Case studies are carried out in this section to validate the proposed coverage algorithm through numerical simulations.
The classic location optimization problem is taken into account with the function $f({p}_i,q)=\|{p}_i-q\|^2$.
The proposed control approach consists of partition dynamics and distributed control strategy. The first part allows to partition the region into subregions with equal workload, while the second part enables agents to arrive at the centroid of sub-regions. As described in Table~\ref{tab:cam}, the $i$-th agent moves towards its centroid $\mathbf{c}_{E_i}$ while taking online partition to balance the load. The evolution of $\mathbf{c}_{E_i}$ depends on the phase angle $\phi_i$, while ${p}_i$ is independent of $\phi_i$.
In addition, the proposed coverage control algorithm is implemented in MATLAB, and simulation results are illustrated in Fig.~\ref{ac} with 4 snapshots displaying the simulation process of $8$ mobile agents. In Fig.~\ref{ac}, blue circles denote the mobile agents, and blue lines refer to the common boundary between agents. Green stars refer to the subregion centroids. The coverage region $\Omega$ is annular, and the polar equations of two closed curves are given
as follows: ${r_{in}}(\theta)=1+0.5\sin{2\theta}$ and ${r_{out}}(\theta)=3+0.5\cos{2\theta}$. Moreover, the workload density function is given by
$\rho(r,\theta) =e^{({{{\sin }^2}\theta+\cos\theta})}+0.01r$. Other parameters are empirically specified as follows: $\kappa_p=0.1$, $\kappa_{\varphi}=0.03$. Note that the parameter setting of control gains is closely related to the workload density function.

In the simulation, multi-agent positions and phase angles of partition bars are initialized at random. At the time $t=4$s, the operation of workload partition goes on and agents start moving towards the centroid of subregions. At the time $t=8$s, the agents get close to the centroid of subregion further.
The centroid of subregions evolves with the partition bars. At the time $t=12$s, each agent reaches its own centroid of subregion and the partition bars almost converge to the desired configuration. As observed in Fig.~\ref{ac}, each agent can eventually move to the centroid of its own subregion. Figure~\ref{workphase} presents time evolution of four key variables on the coverage control algorithm. The first variable is related to the control input for workload partition (i.e., $\dot{\varphi}$), which converges to zero as time approaches the infinity. The second variable is the Euclidean norm of velocity vector (i.e., $\|\dot{p}\|$) and it converges to zero as well. The third one describes the workload on each subregion, and they converge to the same value in the end. The last one represents the phase angle of partition bar, and these phase angles converge to the different values. The above simulation results demonstrate that
the proposed coverage control algorithm works well in the optimization of service cost and workload partition.

\section{Conclusions}\label{sec:con}
This paper addressed the distributed coverage problem of multi-agent systems in uncertain environment, and a novel coverage formulation was developed to minimize the service cost for random events with load balancing. Based on the coverage formulation, a distributed control strategy was designed to drive multi-agent systems towards a desired configuration. In addition, a search algorithm was proposed to identify the optimal configuration with theoretical guarantee. Moreover, case studies were carried out with numerical simulations to substantiate the effectiveness of the proposed
control approach. Future work may include resilient control design of multi-agent systems with nonholonomic constraints against disruptive disturbances.

\section*{Acknowledgment}
The Project was supported by the Fundamental Research Funds for the Central Universities, China University of Geosciences~(Wuhan).


\end{document}